\documentclass{amsart}
\usepackage[T1]{fontenc}
\usepackage{amsmath,amssymb,amsthm}
\usepackage[hidelinks,pdfusetitle]{hyperref}
\usepackage[left=2cm,right=2cm,bottom=3cm,top=3cm]{geometry}
\usepackage{enumitem}
\usepackage{needspace}
\newtheorem{theorem}{Theorem}[section]
\newtheorem{proposition}[theorem]{Proposition}
\newtheorem{lemma}[theorem]{Lemma}
\newtheorem{corollary}[theorem]{Corollary}
\theoremstyle{definition}
\newtheorem{definition}[theorem]{Definition}
\newtheorem{remark}[theorem]{Remark}
\newtheorem{example}[theorem]{Example}
\DeclareMathOperator{\Prim}{Prim}
\DeclareMathOperator{\rank}{rank}
\DeclareMathOperator{\Tr}{Tr}
\DeclareMathOperator{\Her}{Her}
\newcommand{\N}{\mathbb N}
\newcommand{\C}{\mathbb C}
\newcommand{\rkP}{\operatorname{rk}_{\mathrm P}}
\newcommand{\rkF}{\operatorname{rk}_{\mathrm F}}
\newcommand{\ellP}{\ell_{\mathrm P}}
\newcommand{\K}{\mathcal K}
\newcommand{\norm}[1]{\left\lVert#1\right\rVert}
\begin{document}
\title{A new characterization of (pre)liminary C*-algebras}
\author{Martino Lupini}
\address{Dipartimento di Matematica, Universit\`a di Bologna,
Piazza di Porta S. Donato, 5, 40126 Bologna, Italy}
\email{martino.lupini@unibo.it}
\thanks{The author was partially supported by the Marsden Fund Fast-Start
Grant VUW1816 and the Rutherford Discovery Fellowship VUW2002
\textquotedblleft Computing the Shape of Chaos\textquotedblright\ from the
Royal Society of New Zealand. The author was also supported by the
European Research Council via the ERC Starting Grant
\textquotedblleft Definable Algebraic Topology\textquotedblright\ (DAT,
grant agreement No.~101077154). Part of this work was done during a visit to
Chalmers University of Technology and the University of Gothenburg.
Their hospitality is gratefully acknowledged.}
\date{November 12, 2026}
\begin{abstract}
We characterize separable liminary and preliminary C*-algebras by ordinal
ranks and ideal filtrations. Here liminary means that every irreducible
image consists of compact operators, and preliminary means that every
irreducible representation is finite dimensional. The Pedersen rank
stratifies the spectrum by successive local constancy of representation
dimension; Fell rank minimizes the analogous local ranks arising from
positive elements. We prove that a separable type I C*-algebra is liminary
if and only if its Fell rank is countable. A separable C*-algebra is
preliminary if and only if it has a countable continuous ideal filtration
whose successive quotients are countable direct sums of homogeneous
algebras of finite representation dimension. The least possible
filtration length is its Pedersen length. We also determine the possible
ordinal values of the two ranks. Every nonzero countable ordinal is a
Pedersen rank, whereas the nonzero countable Fell ranks are exactly
$\omega\alpha+1$ and $\omega\lambda$, where $\alpha$ is countable and
$\lambda$ is a nonzero countable limit ordinal. All these values occur on
separable preliminary approximately finite-dimensional algebras with
countable Hausdorff spectra. Among nonzero separable unital algebras,
the possible Pedersen ranks of preliminary algebras are exactly the
countable successors, and the possible Fell ranks of liminary algebras
are exactly $\omega\alpha+1$ for countable $\alpha$.
\end{abstract}
\maketitle

\section{Introduction}\label{sec:intro}

The structure of a type I C*-algebra is closely tied to the topology of
its spectrum and the dimensions of its irreducible representations.
Two finiteness conditions are particularly natural. A C*-algebra is
\emph{liminary} if every irreducible representation has image consisting
of compact operators; it is \emph{preliminary} if every irreducible
representation is finite dimensional. Preliminary algebras, also called
FDI algebras, are therefore liminary. We use the standard representation
theory and terminology recalled in
\cite[IV.1.3--IV.1.4]{blackadar_operator_2006}.

The simplest finite-dimensional representation theory occurs for
$n$-homogeneous algebras, whose irreducible representations all have
dimension $n$. Fell's structure theorem describes these algebras as
sections of locally trivial matrix-algebra bundles
\cite[Theorem~3.2, p.~249]{fell_structure_1961}.
For a general preliminary algebra, the dimension function can vary,
and even a uniform bound on its values does not make it locally constant.
The resulting dimension drops suggest studying the spectrum in layers,
successively removing the regions on which the dimension is locally
constant. This paper measures the depth of that process and asks how
the resulting ordinal invariants reflect the algebra's representation
theory. We obtain a parallel description of liminary algebras by applying
the same process to the ranks of positive elements on open parts of the
spectrum, and determine the ordinal values that these invariants can take.

There is a classical local model for the latter construction.
Fell's condition requires that, near each irreducible representation,
some positive element be represented by rank-one projections
\cite[IV.1.4.17]{blackadar_operator_2006}.
Enders and Shulman introduce higher-order versions by placing a uniform
finite bound on the ranks of such local elements; they characterize the
subhomogeneity of central sequence algebras in these terms
\cite[\S3.1 and Theorem~22]{endersCommutativityCentralSequence2022}.
The ordinal invariant considered here records how a finite-valued rank
function changes through successive strata. Thus its role is different
from that of a uniform numerical bound on the values of the function.

We first describe the two sources of rank functions and fix the notation
used in the main results. For a type I algebra $A$, we identify its
spectrum $\widehat A$ with its primitive ideal space, endowed with the
Jacobson topology. The dimension function is
$w_A([\pi])=\dim(\pi)$, whereas a positive element $a\in A_+$ gives
$w_a([\pi])=\rank(\pi(a))$.
Fell's analysis of limits of finite-dimensional representations constrains
the accumulation of dimension values
\cite[Corollary~1 to Theorem~2.2, p.~388]{fellDualSpacesCAlgebras1960}.
The trace inequality of Archbold and Somerset
\cite[Theorem~2.4, p.~89]{archbold_transition_1993} gives the corresponding
control for positive elements. We use these facts to treat $w_A$ and $w_a$
as weights on the spectrum; Section~\ref{sec:spectra} gives the details.

For a fully supported finite weight $w$ on a second countable, locally
compact, hereditarily Baire $T_0$ space $X$, remove its locus of local
constancy, repeat on the remaining closed subspace, and take
intersections at limit stages. If a point $t$ is removed from the
$\alpha$-th derivative, assign it point rank $\omega\alpha+w(t)$.
Write $r(X,w)$ for the supremum of these point ranks and $o(X,w)$ for
the first stage at which the derivative is empty. For preliminary $A$,
its \emph{Pedersen rank} and \emph{Pedersen length} are
\[
 \rkP(A)=r(\widehat A,w_A),\qquad
 \ellP(A)=o(\widehat A,w_A).
\]
In particular, these invariants use representation dimension itself,
rather than a supremum over the element functions $w_a$.

For a type I algebra $A$, a \emph{local section} at
$t\in\widehat A$ is a pair $(U,a)$, where $U$ is an open neighborhood
of $t$ and $a\in A_+$ satisfies $0<w_a(s)<\infty$ for all $s\in U$.
The \emph{Fell rank} at $t$ and the Fell rank of $A$ are
\[
 \rkF(t,A)=\inf_{(U,a)}r(U,w_a|_U),\qquad
 \rkF(A)=\sup_{t\in\widehat A}\rkF(t,A).
\]
An empty infimum has value $\infty$, a symbol larger than every ordinal.
For nonzero algebras, Fell rank one is exactly Fell's condition;
the zero algebra has rank zero. The local choices in this definition
allow one to recover liminarity without imposing a uniform finite bound
on representation dimensions or on the ranks of all local sections.

\begin{theorem}\label{thm:intro-ccr}
A separable type I C*-algebra $A$ is liminary if and only if
$\rkF(A)$ is a countable ordinal.
\end{theorem}
Theorem~\ref{thm:ccr} proves this characterization. For a liminary algebra,
functional-calculus cutdowns of positive elements supply finite-rank
local sections. Conversely, each local section is realized as
the dimension-weighted spectrum of a full hereditary preliminary
subalgebra of the corresponding ideal. This forces that ideal to be
liminary. The local ordinal bounds then combine to a single countable
bound by second countability.

For preliminary algebras, the dimension derivative has a direct ideal
interpretation. Theorem~\ref{thm:preliminary} shows that a separable
C*-algebra $A$ is preliminary if and only if it admits a countable
continuous increasing filtration of closed ideals from $0$ to $A$ whose
successive quotients are countable $c_0$-direct sums of homogeneous
algebras of finite representation dimension. The dimensions may vary
between summands. The Pedersen length is the least possible length of
such a filtration:
\[
 \ellP(A)=\min\{\delta<\omega_1:
       \text{such a filtration of length }\delta\text{ exists}\}.
\]
The strata of the dimension derivative give a filtration attaining this
minimum. Here continuity means
$I_\lambda=\overline{\bigcup_{\alpha<\lambda}I_\alpha}$ at limit ordinals.
This description measures the number of layers needed to build the
algebra, with each layer a countable direct sum of homogeneous algebras
and with continuity at limit stages.

These filtrations belong to the classical framework of composition
series \cite[IV.1.1.10--IV.1.1.12]{blackadar_operator_2006}.
In particular, type I algebras admit composition series with
continuous-trace subquotients
\cite[IV.1.4.25--IV.1.4.28]{blackadar_operator_2006}.
Here the dimension derivative gives a specific filtration with layers
of the form above, and the minimum formula identifies the least length
among all filtrations with these allowed layers.

\Needspace{6\baselineskip}
At finite levels, Pedersen rank detects local constancy as well as the
size of irreducible representations.
\begin{theorem}\label{thm:intro-finite}
For a nonzero separable preliminary algebra $A$ and $n\geq1$,
$\rkP(A)\leq n$ if and only if its irreducible-dimension function is
locally constant and bounded by $n$.
\end{theorem}
Equivalently, $A$ is a finite direct sum of homogeneous algebras whose
representation dimensions are at most $n$; see
Proposition~\ref{prop:finite-pedersen}.
This is stronger than classical $n$-subhomogeneity, which only bounds
irreducible dimensions \cite[IV.1.4.1]{blackadar_operator_2006}.
The examples in Section~\ref{sec:finite} distinguish both this condition
and the two sources of rank functions: a $2$-subhomogeneous algebra can
have Pedersen rank $\omega+1$, while an element of $C(X,M_2)$ can have
element-weight rank $\omega+1$ even though the algebra has Pedersen
rank $2$.

\Needspace{12\baselineskip}
The ordinal ranges of the invariants are already realized by concrete
approximately finite-dimensional (AF) algebras.
\begin{theorem}\label{thm:intro-realization}
Every nonzero countable ordinal is the Pedersen rank of a separable
preliminary AF algebra. The possible nonzero countable Fell ranks are
\[
 \{\omega\alpha+1:\alpha<\omega_1\}\ \cup\
 \{\omega\lambda:0<\lambda<\omega_1,\ \lambda\text{ is limit}\}.
\]
They all occur on separable preliminary AF algebras with countable
Hausdorff spectrum. For every $\alpha<\omega_1$ there is a unital such
algebra with $\ellP(A)=\alpha+1$ and
$\rkP(A)=\rkF(A)=\omega\alpha+1$.
\end{theorem}
The construction starts with $\C$ and iterates scalar unitization of
countable direct sums of matrix amplifications. A distinguished
character survives to the prescribed derivative stage and forces the
required Fell rank. Theorem~\ref{thm:successor-realization} and
Corollaries~\ref{cor:all-pedersen}--\ref{cor:all-fell} give the construction
and the exact rank computations. These examples are scattered.
Compactness imposes additional restrictions: for a nonzero unital
preliminary algebra, the Pedersen length and rank are nonzero countable
successors, while the Fell rank of a nonzero unital liminary algebra is
$\omega\alpha+1$ for a countable ordinal $\alpha$.
Every such unital value occurs
(Corollary~\ref{cor:unital-realizations}).
The two ranks can also be far apart: every nonzero countable ordinal
occurs as the Pedersen rank of a preliminary AF algebra of Fell rank one
(Corollary~\ref{cor:independent-pedersen}).

Throughout, all C*-algebras are separable and all irreducible
representations are nonzero and nondegenerate. Operator ranks and
representation dimensions take values in $\N\cup\{\infty\}$, with $0\in\N$; products
$\omega\alpha$ denote ordinal multiplication on the left.

Section~\ref{sec:weights} develops the weighted-space derivatives and
ranks. Section~\ref{sec:spectra} applies them to spectra and local
sections, and Section~\ref{sec:canonical} defines the Pedersen
and Fell invariants and proves their basic properties.
Section~\ref{sec:characterization} proves the characterizations and their
ideal and compact-spectrum consequences. Section~\ref{sec:finite}
treats finite ranks and dimension drops, and
Section~\ref{sec:realization} constructs the algebras realizing the
possible countable ranks.

\section{Weighted spaces and ranks}\label{sec:weights}

The invariants in this paper depend on where a rank function is locally
constant, rather than only on the size of its values. We develop the
weighted-space derivative that records this distinction and establish
the locality and finite-range estimates needed to pass from spectra to
algebraic rank bounds.

\begin{definition}\label{def:fell-space}
A \emph{Fell space} is a second countable, locally compact, hereditarily
Baire $T_0$ space. A \emph{weight} on a second countable space $X$ is a
function $w:X\to\N\cup\{\infty\}$ such that for every sequence $(x_i)$
with set of limits $L$,
\begin{equation}\label{eq:weight}
 \sum_{x\in L}w(x)\leq\liminf_iw(x_i).
\end{equation}
The sum is the supremum of finite subsums. The weight is \emph{finite}
if it has values in $\N$, and \emph{fully supported} if it is strictly
positive everywhere. A \emph{preliminary space} is a Fell space with
a fully supported finite weight.
\end{definition}

Local compactness means that each point has a basis of compact
neighborhoods, not necessarily closed. A weight is lower
semicontinuous: otherwise a countable neighborhood basis gives a
convergent sequence violating \eqref{eq:weight}.
Its support $X_w=\{x:w(x)>0\}$ is open. Restriction to any subspace
preserves the weight inequality. Open and closed subspaces of a Fell
space are again Fell spaces.

\begin{lemma}\label{lem:locally-constant-polish}
A preliminary space with locally constant weight is locally compact
Polish.
\end{lemma}
\begin{proof}
If a sequence has distinct limits $x,y$, local constancy at $x$
makes its weight eventually $w(x)$. Then \eqref{eq:weight} gives
$w(x)+w(y)\leq w(x)$, a contradiction.
A first countable non-Hausdorff space has two points whose basic
neighborhoods intersect at every scale; choosing points in these
intersections gives a sequence with both limits.
Thus the space is Hausdorff. A locally compact second countable
Hausdorff space is Polish.
\end{proof}

For a preliminary space $(X,w)$ and a subspace $Y$ put
$Y_{\mathrm{lc}}=\{y\in Y:w|_Y\text{ is locally constant at }y\}$.
Define
\begin{equation}\label{eq:derivative}
 \partial_0(X,w)=X,\qquad
 \partial_{\alpha+1}(X,w)=\partial_\alpha(X,w)
       \setminus(\partial_\alpha(X,w))_{\mathrm{lc}},\qquad
 \partial_\lambda(X,w)=\bigcap_{\alpha<\lambda}\partial_\alpha(X,w)
\end{equation}
at limit ordinals, and set
$\sigma_\alpha(X,w)=\partial_\alpha(X,w)\setminus
\partial_{\alpha+1}(X,w)$. Suppress $w$ when understood.

\begin{lemma}\label{lem:exhaustion}
The derivative exhausts $X$ at a countable stage. Its strata are
locally compact Polish spaces with locally constant weight.
\end{lemma}
\begin{proof}
For a nonempty closed derivative $Y$, the closed sets $\{w|_Y\leq n\}$
cover the Baire space $Y$. One has nonempty relative interior $V$.
Let $m$ be the largest value of $w$ on $V$.
Then $V\cap\{w>m-1\}$ is a nonempty relatively open set of constant
weight. Thus $Y_{\mathrm{lc}}$ is nonempty and open in $Y$.

At every nonempty derivative choose a basic open set $B_\alpha$
meeting $\partial_\alpha X$ but missing $\partial_{\alpha+1}X$.
For $\alpha<\beta$, $B_\alpha$ misses $\partial_\beta X$, whereas
$B_\beta$ meets it. These basic open sets are distinct.
A countable basis therefore forces exhaustion before $\omega_1$.
Each stratum is open in a closed Fell subspace, so
Lemma~\ref{lem:locally-constant-polish} applies.
\end{proof}

\begin{definition}\label{def:canonical-space-rank}
For $x\in\sigma_\alpha(X,w)$ define
\[
 o_{(X,w)}(x)=\alpha,\qquad r_{(X,w)}(x)=\omega\alpha+w(x).
\]
The \emph{length} and \emph{rank} of $(X,w)$ are
\[
 o(X,w)=\min\{\alpha:\partial_\alpha(X,w)=\varnothing\},
 \qquad r(X,w)=\sup_{x\in X}r_{(X,w)}(x).
\]
Both invariants of the empty space are $0$.
For a finite weight not fully supported, compute them on $X_w$.
\end{definition}

For nonempty $X$, length and rank are nonzero countable ordinals.
If $o(X,w)=\beta+1$, the point-rank definition gives
\begin{equation}\label{eq:last-stratum}
 r(X,w)=\omega\beta+
   \sup\{w(x):x\in\sigma_\beta(X,w)\}.
\end{equation}
If the length is limit, $r(X,w)=\omega o(X,w)$, because the point
orders are cofinal in it and each point of order $\gamma$ has rank
below $\omega(\gamma+1)$. In particular
\begin{equation}\label{eq:length-bounds}
 r(X,w)\leq\omega o(X,w),\qquad
 \omega\beta+1\leq r(X,w)\leq\omega(\beta+1)
 \quad\text{if }o(X,w)=\beta+1.
\end{equation}

\begin{lemma}[Locality]\label{lem:open-locality}
For an open subset $U$ with the restricted weight,
\[
 \partial_\alpha U=U\cap\partial_\alpha X,\qquad
 \sigma_\alpha U=U\cap\sigma_\alpha X.
\]
Point orders and ranks agree on $U$.
If $Y\subseteq X$ is itself a Fell space with the restricted weight,
then $\partial_\alpha Y\subseteq Y\cap\partial_\alpha X$;
hence $o(Y)\leq o(X)$ and $r(Y)\leq r(X)$.
\end{lemma}
\begin{proof}
Local constancy is local in an open ambient subspace.
Induction gives the equalities, with intersections at limits.
For a general subspace, induction uses the fact that local constancy
on an ambient derivative implies local constancy on its intersection
with the smaller derivative. Point orders can only decrease, while
point weights remain the same.
\end{proof}

\begin{lemma}\label{lem:rank-neighborhood}
For each $x\in X$ there is an open neighborhood $U$ with
$r(U,w|_U)=r_{(X,w)}(x)$.
\end{lemma}
\begin{proof}
Let $\alpha=o_{(X,w)}(x)$ and $n=w(x)$.
Choose $U$ with
$U\cap\partial_\alpha X\subseteq\sigma_\alpha X\cap w^{-1}(\{n\})$.
By locality all points of $U$ have order at most $\alpha$.
Those of order $\alpha$ have weight $n$, and all others have rank
at most $\omega\alpha$. Thus $r(U)\leq\omega\alpha+n$.
The reverse inequality follows from the rank of $x$ in $U$.
\end{proof}

\begin{proposition}[Finite-range bound]\label{prop:finite-range}
If the values of $w$ belong to $\{v_1<\cdots<v_N\}$, then
\[
 o(X,w)\leq N,\qquad r(X,w)\leq\omega(N-1)+v_1.
\]
For $n\geq1$, $r(X,w)\leq n$ exactly when $w$ is locally constant
and bounded by $n$.
\end{proposition}
\begin{proof}
In any subspace the locus of maximum weight is open and has constant
weight, so disappears at the next derivative.
After $N-1$ steps only weight $v_1$ can remain.
If the process stops earlier its rank is at most $\omega(N-1)$;
otherwise use \eqref{eq:last-stratum}.
A point of positive order has rank at least $\omega+1$.
Thus finite rank forces order zero everywhere, or equivalently local
constancy. The rank is then $\sup_Xw$.
\end{proof}

\begin{remark}\label{rem:no-weight-monotonicity}
Restriction keeps the weight fixed. Pointwise comparison of different
weights does not compare their ranks. On the convergent-sequence
space $X=\{0\}\cup\{1/n:n\geq1\}$, the weight $u(0)=1$,
$u(1/n)=2$ satisfies $u\leq v=2$, but
$r(X,u)=\omega+1>2=r(X,v)$.
\end{remark}

\section{Spectra and local sections}\label{sec:spectra}

To apply these derivatives to C*-algebras, we need both the dimension
function and the ranks of positive elements to satisfy the weight
inequality. We verify this from the trace estimate, then use hereditary
subalgebras to realize the finite weights occurring on local sections as
representation dimensions.

For a type I algebra $A$ we identify $\widehat A$ with $\Prim(A)$
by the kernel map \cite[IV.1.3.6]{blackadar_operator_2006}, and use the
Jacobson topology.
In general, the topology on $\widehat A$ is the inverse-image topology
of this map. For type I algebras it is a homeomorphism.
For an ideal $I$, $\widehat I$ is the corresponding open subspace,
and $\widehat{A/I}$ is its closed complement.
The primitive spectrum is second countable, locally compact, and $T_0$
\cite[II.6.5.6--II.6.5.7]{blackadar_operator_2006}.
More precisely, the compact norm-level sets
$\{[\pi]:\norm{\pi(a)}\geq\varepsilon\}$, with $a\in A$ and
$\varepsilon>0$, give the required compact neighborhoods
\cite[II.6.5.6]{blackadar_operator_2006}: for a point $t$ in an open
set $U$, choose $a\in I_U$ and $\varepsilon>0$ with
$\norm{\pi_t(a)}>\varepsilon$. The corresponding compact set lies
in $U$ and contains an open neighborhood of $t$.
Every closed subspace is the primitive spectrum of a quotient
\cite[II.6.5.4]{blackadar_operator_2006}, so the Baire theorem
\cite[II.6.5.14]{blackadar_operator_2006} makes it hereditarily Baire.

We use the trace estimate for distinct limits $t_1,\ldots,t_k$ of
a sequence $([\pi_i])$:
\begin{equation}\label{eq:trace}
 \sum_{j=1}^k\Tr(\pi_{t_j}(b))
       \leq\liminf_i\Tr(\pi_i(b)),\qquad b\in A_+.
\end{equation}
This follows directly from \cite[Theorem 2.4, p.~89]{archbold_transition_1993}:
that theorem bounds the sum over the entire set of limits of a convergent
net, and hence over each finite subset. A sequence with a nonempty limit
set is such a net; for a sequence without limits, the left-hand side
of the weight inequality is zero.

\begin{proposition}\label{prop:representation-weights}
For type I $A$ and $a\in A_+$ the function
\[
 w_a:\widehat A\to\N\cup\{\infty\},\qquad
 w_a([\pi])=\rank(\pi(a))
\]
is a weight. For a strictly positive $h\in A$ one has
$w_h([\pi])=\dim(\pi)$, so the dimension function $w_A$ is a weight.
\end{proposition}
\begin{proof}
Put $f_m(s)=\min\{1,ms\}$ for $s\geq0$.
Then $f_m(a)\in A_+$; in each representation its images increase
strongly to the support projection of $\pi(a)$.
Their traces increase to $w_a([\pi])$ and do not exceed it.
For each finite subset $F$ of the limits of a sequence,
\[
 \sum_{t\in F}\Tr(\pi_t(f_m(a)))
 \leq\liminf_i\Tr(\pi_i(f_m(a)))
 \leq\liminf_iw_a([\pi_i])
\]
by \eqref{eq:trace}. Let $m$ tend to infinity, and take the supremum
over finite $F$ to obtain \eqref{eq:weight}.
A strictly positive element has support projection the identity in
every nondegenerate representation. Such elements exist by
separability \cite[II.4.2.1--II.4.2.4]{blackadar_operator_2006}.
The dimension case is classical; see
\cite[Corollary~1 to Theorem~2.2, p.~388]{fellDualSpacesCAlgebras1960}.
\end{proof}

\begin{definition}\label{def:local-section}
For type I $A$ and $t\in\widehat A$, a \emph{local section}
at $t$ is a pair $(U,a)$ with $U$ an open neighborhood of $t$,
$a\in A_+$, and $0<w_a(s)<\infty$ for $s\in U$. Set
\[
 r(U,a)=r(U,w_a|_U),\qquad o(U,a)=o(U,w_a|_U).
\]
\end{definition}

The next realization allows a local section to be replaced by an
element with prescribed support. It uses equality of weights rather
than pointwise domination.

\begin{lemma}[Hereditary realization]\label{lem:hereditary-section}
Let $(U,a)$ be a local section, let $I$ be the ideal
corresponding to $U$, and put $B=I\cap\Her(a)$, where
$\Her(a)=\overline{aAa}$. Then $B$ is preliminary,
$\widehat B$ is canonically homeomorphic to $U$, and $w_B=w_a|_U$.
For a strictly positive $b\in B$, regarded as an element of $A$,
\[
 w_b|_U=w_a|_U,\qquad w_b|_{\widehat A\setminus U}=0.
\]
\end{lemma}
\begin{proof}
Let $J_a$ be the ideal generated by $a$. The nonvanishing of $a$
throughout $U$ gives $I\subseteq J_a$ by the ideal correspondence.
The hereditary ideal correspondence identifies $B$ as a full
hereditary subalgebra of $I$
\cite[II.5.3.5]{blackadar_operator_2006}.
Indeed, $E=\overline{BI}$ is a $B$--$I$ imprimitivity bimodule with
inner products ${}_B\langle x,y\rangle=xy^*$ and
$\langle x,y\rangle_I=x^*y$. Heredity gives
$\overline{\operatorname{span}}EE^*=B$, and fullness gives
$\overline{\operatorname{span}}E^*E=I$.
Thus $B$ and $I$ are Morita equivalent, and the representation
correspondence identifies $\widehat B$ with $U$
\cite[II.7.6.13]{blackadar_operator_2006}.

For $[\pi]\in U$, the essential space for $\pi(\Her(a))$ is
$\overline{\pi(a)H_\pi}$. The restriction to $I$ is nondegenerate.
An approximate identity $(e_i)$ of $I$ gives
$\pi(ae_ia)\to\pi(a^2)$ strongly, and $ae_ia\in B$.
Hence the essential space of $\pi|_B$ is also
$\overline{\pi(a)H_\pi}$. The representation on this essential space is irreducible
\cite[II.6.1.9]{blackadar_operator_2006}; its dimension is
$w_a([\pi])$, which is finite. A strictly positive $b\in B$ has that whole
essential space as its support. Outside $U$ the ideal $I$ vanishes.
\end{proof}

\section{Pedersen rank, Pedersen length, and Fell rank}\label{sec:canonical}

The two ways of assigning weights to a spectrum lead to different
invariants. We now define the Pedersen and Fell ranks and establish the
properties that allow the weighted-space arguments to be used for
ideals, their filtrations, and local sections.

\subsection{Pedersen rank and length}

Applying the weighted-space derivative to representation dimension gives
strata on which dimension is locally constant. Their homogeneous pieces
determine an ideal filtration of a preliminary algebra, preparing the
characterization in Section~\ref{sec:characterization}.

\begin{definition}\label{def:pedersen}
For preliminary $A$, give $\widehat A$ its dimension weight
$w_A([\pi])=\dim(\pi)$. Its \emph{Pedersen length} and
\emph{Pedersen rank} are
\[
 \ellP(A)=o(\widehat A,w_A),\qquad
 \rkP(A)=r(\widehat A,w_A).
\]
At a point $t$ put
\[
 \ellP(t,A)=o_{(\widehat A,w_A)}(t),\qquad
 \rkP(t,A)=\omega\ellP(t,A)+w_A(t).
\]
The dimension-weight derivative is the Pedersen derivative of $A$.
\end{definition}

The Pedersen invariants use the dimension weight, rather than a
supremum over element weights. For a positive element with finite weight
we also write
\[
 r(w_a)=r(X_{w_a},w_a|_{X_{w_a}}),\qquad
 o(w_a)=o(X_{w_a},w_a|_{X_{w_a}})
\]
for its element-weight rank and length, respectively.

\begin{proposition}\label{prop:element-hereditary}
If $A$ is type I, $a\in A_+$, and $w_a$ is finite, then
\[
 r(w_a)=\rkP(\Her(a)),\qquad o(w_a)=\ellP(\Her(a)).
\]
For strictly positive $h$ in preliminary $A$ this gives
$r(w_h)=\rkP(A)$ and $o(w_h)=\ellP(A)$.
\end{proposition}
\begin{proof}
The hereditary representation correspondence identifies
$\widehat{\Her(a)}$ with $X_{w_a}$ and its dimension function with
$w_a$. Apply Definition~\ref{def:pedersen}.
\end{proof}

\begin{proposition}\label{prop:pedersen-locality}
If $I$ is an ideal of preliminary $A$, the pointwise Pedersen orders
and ranks of $I$ are the restrictions of those of $A$. Moreover,
\[
 \ellP(I),\ellP(A/I)\leq\ellP(A),\qquad
 \rkP(I),\rkP(A/I)\leq\rkP(A).
\]
\end{proposition}
\begin{proof}
Restriction to an ideal and passage to a quotient preserve the
dimensions of the corresponding representations. Apply
Lemma~\ref{lem:open-locality}.
\end{proof}

\begin{proposition}\label{prop:homogeneous-filtration}
For preliminary $A$, with $X=\widehat A$ and $\delta=\ellP(A)$,
there is a continuous increasing chain of ideals
$(I_\alpha)_{\alpha\leq\delta}$, from $0$ to $A$, with
\[
 \widehat{I_{\alpha+1}/I_\alpha}=\sigma_\alpha(X,w_A).
\]
Each subquotient is a countable direct sum of homogeneous algebras.
Furthermore, $\ellP(A)=1$ precisely when $A\ne0$ and $w_A$ is
locally constant.
\end{proposition}
\begin{proof}
Let $I_\alpha$ correspond to $X\setminus\partial_\alpha(X,w_A)$.
At a limit these open sets are the union of the earlier ones,
giving norm closure of the union of ideals.
On each stratum the dimension function is locally constant.
Its individual level sets are disjoint clopen subspaces,
whose ideals are homogeneous; their direct sum is the subquotient.
The last assertion follows from the first derivative.
\end{proof}

\subsection{Fell rank}

Fell rank minimizes the weighted-space ranks of local sections near each
representation. This gives an invariant adapted to liminarity. We relate
it to Fell's condition and use a rank-one reduction to restrict its
possible ordinal values.

\begin{definition}\label{def:fell-rank}
For type I $A$ set
\[
 \rkF(t,A)=\inf\{r(U,a):(U,a)\text{ is a local section at }t\},
 \qquad
 \rkF(A)=\sup_{t\in\widehat A}\rkF(t,A).
\]
An empty infimum is $\infty$, an additional symbol larger than every
ordinal; any supremum involving it is $\infty$.
The Fell rank of the zero algebra is $0$.
\end{definition}

The nonempty infima are attained, since section ranks are countable
ordinals. Lemma~\ref{lem:rank-neighborhood} shows that the same infimum
is obtained by minimizing $\omega o_{(U,w_a)}(t)+w_a(t)$.

For completeness, Fell's condition is equivalent to $\rkF(A)\leq1$.
The functional-calculus normalization used below also appears in
\cite[Remark~11]{endersCommutativityCentralSequence2022}.
Suppose first that $A\ne0$ and $\rkF(A)\leq1$.
At each $t$, the attained minimum gives a section $(U,a)$ of rank one,
so $w_a=1$ throughout $U$. Choose
$0<\varepsilon<\norm{\pi_t(a)}$. Lower semicontinuity of the norm
function makes
\[
 V=U\cap\{s:\norm{\pi_s(a)}>\varepsilon\}
\]
an open neighborhood of $t$. Take a continuous function
$f:[0,\norm a]\to[0,1]$ with $f(0)=0$ and $f=1$ on
$[\varepsilon,\norm a]$. Then $f(a)\in A_+$, and each
$\pi_s(f(a))$, $s\in V$, is a rank-one projection: the positive
rank-one operator $\pi_s(a)$ has its unique nonzero eigenvalue
$\norm{\pi_s(a)}>\varepsilon$. This is Fell's condition
\cite[IV.1.4.17]{blackadar_operator_2006}.
Conversely, its local rank-one projections give sections of rank one.
The zero algebra satisfies Fell's condition vacuously and has Fell
rank zero. Thus nonzero algebras satisfy Fell's condition exactly
when their Fell rank is one.

\begin{proposition}\label{prop:fell-locality}
For an ideal $I$ of type I $A$, the Fell point ranks agree with those
of $A$ on $\widehat I$. For preliminary $A$,
\[
 \rkF(t,A)\leq\rkP(t,A),\qquad \rkF(A)\leq\rkP(A).
\]
\end{proposition}
\begin{proof}
Sections formed in $I$ are sections in $A$.
Conversely, restrict an ambient section to its intersection with
$\widehat I$ and use Lemma~\ref{lem:hereditary-section} to realize
the restricted weight exactly by an element of $I$.
Restriction does not increase rank. For the second assertion, use
a strictly positive element and
Lemma~\ref{lem:rank-neighborhood} for the dimension weight.
\end{proof}

\begin{lemma}[Rank-one reduction]\label{lem:rank-one-reduction}
If $(U,a)$ is a local section at $t$ and
$\alpha=o_{(U,w_a)}(t)$, then $\rkF(t,A)\leq\omega\alpha+1$.
\end{lemma}
\begin{proof}
Realize $(U,w_a)$ as the dimension-weighted spectrum of the
hereditary preliminary algebra $B$ of
Lemma~\ref{lem:hereditary-section}. Work inside $B$.
Let $F=\partial_\alpha(\widehat B,w_B)$ and let $J$ be the ideal
with spectrum $\widehat B\setminus F$.

The point $t$ has a relatively open neighborhood in $F$ of constant
dimension $n=w_B(t)$. Its ideal in $B/J$ is $n$-homogeneous.
By the local structure of homogeneous algebras
\cite[Theorem 3.2]{fell_structure_1961}, a smaller neighborhood $V$
has ideal $C_0(V,M_n)$. Equivalently, homogeneous algebras have
continuous trace and local rank-one projections
\cite[IV.1.4.14--IV.1.4.17]{blackadar_operator_2006}.
Multiply a constant rank-one projection by a nonnegative function in
$C_c(V)$ positive at $t$, and extend by zero in $B/J$, obtaining
$p\in(B/J)_+$. Its ranks on $F$ are zero or one, with rank one at $t$.

Choose a positive lift $c\in B$ and a strictly positive $j\in J$,
taking $j=0$ if $J=0$. Put $b=c+j$.
Outside $F$, $j$ is strictly positive in the finite-dimensional
irreducible image, so $w_b=w_B$. On $F$ its weight is zero or one.
Let $U_b$ be its open support, which contains $t$.
On the open set $U_b\setminus F$ we have $w_b=w_B$.
Open locality, first for $(U_b,w_b)$ and then for
$(\widehat B,w_B)$, gives
\[
 \partial_\alpha(U_b,w_b)\cap(U_b\setminus F)
 =\partial_\alpha(U_b\setminus F,w_B)
 =(U_b\setminus F)\cap\partial_\alpha(\widehat B,w_B)
 =\varnothing.
\]
Thus $\partial_\alpha(U_b,w_b)\subseteq F$.
The remaining weight is constantly one, so the next derivative is
empty and $r(U_b,w_b)\leq\omega\alpha+1$.
The positive element $b\in B\subseteq A$ gives the required section.
\end{proof}

\begin{corollary}\label{cor:fell-point-form}
Every ordinal-valued Fell point rank is $\omega\alpha+1$.
Every nonzero countable global Fell rank is $\omega\alpha+1$, or
$\omega\lambda$ for a nonzero countable limit ordinal $\lambda$.
\end{corollary}
The global assertion follows from the pointwise assertion by taking
suprema.
\begin{proof}
A section realizing the minimum point rank has, after shrinking,
rank $\omega\alpha+n$. The preceding lemma replaces it by one
of rank at most $\omega\alpha+1$, so minimality forces $n=1$.
For the global statement take the supremum of $\omega\alpha_t+1$.
If the $\alpha_t$ have a maximum, this gives the first form.
Otherwise their supremum is a nonzero limit $\lambda$, giving
$\omega\lambda$.
\end{proof}

\section{Characterization and rank-bounded ideals}\label{sec:characterization}

The ranks now translate into structural descriptions of the algebra.
We prove the liminary and preliminary characterizations, identify the
least length of an ideal filtration whose layers are countable direct
sums of homogeneous algebras, and examine how compactness and ideal
restriction constrain these invariants. The ideal results also
locate the largest liminary ideal of a type I algebra.

\begin{theorem}\label{thm:ccr}
For a separable type I algebra $A$, the following are equivalent:
\begin{enumerate}[label=\textup{(\arabic*)}]
\item $A$ is liminary;
\item every point of $\widehat A$ admits a local section;
\item $\rkF(A)<\omega_1$.
\end{enumerate}
\end{theorem}
\begin{proof}
Suppose $A$ is liminary and fix $t$.
Lift a rank-one projection in $\pi_t(A)=\K(H_t)$ to $z\in A_+$.
Choose continuous nonnegative $f$ vanishing on $[0,1/2]$, with
$f(1)=1$. Every $\pi_s(z)$ is compact and has only finitely many
eigenvalues above $1/2$, so $\pi_s(f(z))$ has finite rank.
At $t$ it is a rank-one projection. Its open support is a
local section at $t$.

Conversely, a section $(U,a)$ gives the full hereditary preliminary
subalgebra $B=I_U\cap\Her(a)$ of Lemma~\ref{lem:hereditary-section}.
In each irreducible representation of $I_U$, the essential space
of $B$ is finite dimensional, so its image consists of finite-rank
operators. Since $B$ generates $I_U$ as an ideal, the image of
$I_U$ consists of compact operators.
Thus $I_U$ is liminary.
If these neighborhoods cover the spectrum, the corresponding ideals
generate $A$. Every irreducible representation of $A$ maps each
such ideal into the compact operators and hence maps $A$ into them.
This proves (2)$\Rightarrow$(1).

Put $F_\gamma(A)=\{t:\rkF(t,A)\leq\gamma\}$ for $\gamma<\omega_1$.
This set is open: a section witnessing membership is a section at
every point of its domain, with the same rank bound.
Under (2), these sets cover $\widehat A$.
Second countability gives a countable subcover
$(F_{\gamma_i}(A))_i$. For the countable ordinal
$\gamma=\sup_i\gamma_i$ we have $\rkF(A)\leq\gamma$.
Finally (3) implies (2) by definition.
\end{proof}

\Needspace{15\baselineskip}
\begin{theorem}\label{thm:preliminary}
For a separable C*-algebra $A$, the following are equivalent:
\begin{enumerate}[label=\textup{(\arabic*)}]
\item $A$ is preliminary;
\item there are a countable ordinal $\delta$ and closed ideals
$(I_\alpha)_{\alpha\leq\delta}$ with
\[
 I_0=0,\qquad I_\delta=A,\qquad
 I_\lambda=\overline{\bigcup_{\alpha<\lambda}I_\alpha}
 \quad(\lambda\leq\delta\text{ limit}),
\]
forming an increasing chain, such that each
$I_{\alpha+1}/I_\alpha$ is a countable $c_0$-direct sum of
$n$-homogeneous algebras with finite $n\geq1$, allowed to vary
between summands.
\end{enumerate}
When these conditions hold, the Pedersen length and rank
of $A$ are countable, and its Pedersen derivative supplies such a
filtration with $\delta=\ellP(A)$. Moreover,
\[
 \ellP(A)=\min\{\delta<\omega_1:
       \text{a filtration as in \textup{(2)} exists}\}.
\]
\end{theorem}
\begin{proof}
If $A$ is preliminary, Proposition~\ref{prop:representation-weights}
and Lemma~\ref{lem:exhaustion} give countable Pedersen length and
rank. Proposition~\ref{prop:homogeneous-filtration} supplies the
stated continuous ideal filtration.

Conversely, suppose that such a filtration exists and let $\pi$
be an irreducible representation of $A$ on $H_\pi$.
There is a least $\beta\leq\delta$ such that $\pi(I_\beta)\ne0$.
It is nonzero and cannot be a limit ordinal: at a limit, continuity
of the chain would make $\pi$ vanish on $I_\beta$ if it vanished on
all earlier ideals. Thus $\beta=\alpha+1$.
The restriction of $\pi$ to $I_{\alpha+1}$ is irreducible and
nondegenerate on the same Hilbert space
\cite[II.6.1.6--II.6.1.9]{blackadar_operator_2006}, and vanishes on
$I_\alpha$. It therefore induces an irreducible representation of
$I_{\alpha+1}/I_\alpha$. An irreducible representation of a
$c_0$-direct sum is nonzero on exactly one summand. That summand is
$n$-homogeneous for some finite $n$, so $\dim H_\pi=n<\infty$.
Hence $A$ is preliminary.

To prove minimality, put $X=\widehat A$ and
$U_\alpha=\widehat{I_\alpha}\subseteq X$ for any filtration
in \textup{(2)}. We claim that
\[
 \partial_\alpha(X,w_A)\subseteq X\setminus U_\alpha
 \qquad(\alpha\leq\delta).
\]
This is immediate at $\alpha=0$. For the successor step,
$U_{\alpha+1}\setminus U_\alpha$ is open in
$X\setminus U_\alpha$ and has locally constant dimension weight,
because its algebra is a direct sum of homogeneous algebras.
By the induction hypothesis, every point of
$\partial_\alpha(X,w_A)$ in this layer is therefore a point of
local constancy of the restricted weight and is removed at the
next derivative. At a limit $\lambda$, continuity of the ideal
chain gives $U_\lambda=\bigcup_{\alpha<\lambda}U_\alpha$;
intersecting the earlier inclusions proves the claim.
Since $U_\delta=X$, the claim gives $\ellP(A)\leq\delta$.
The canonical filtration attains equality, proving the minimum formula.
\end{proof}

\begin{remark}\label{rem:pedersen-infinity}
The Pedersen invariants are defined here for preliminary algebras.
Extending them to type I algebras by assigning $\infty$ when $w_A$ is
not finite makes countability of either invariant equivalent to
preliminarity.
\end{remark}

For an arbitrary separable algebra, Theorem~\ref{thm:ccr} says that
it is liminary exactly when it is type I and has countable Fell
rank. Theorem~\ref{thm:preliminary} characterizes preliminarity
directly by a continuous ideal filtration.

\begin{proposition}[Compact spectra]\label{prop:compact-spectrum}
Suppose that $A\ne0$ has compact primitive spectrum.
\begin{enumerate}[label=\textup{(\arabic*)}]
\item If $A$ is preliminary, then for some countable ordinal $\alpha$
and integer $d\geq1$,
\[
 \ellP(A)=\alpha+1,\qquad \rkP(A)=\omega\alpha+d.
\]
The last Pedersen derivative is compact and has maximum
representation dimension $d$.
\item If $A$ is liminary, its global Fell rank is attained at a point
and has the form $\omega\beta+1$ for a countable ordinal $\beta$.
\end{enumerate}
In particular, these conclusions apply to nonzero unital algebras.
\end{proposition}
\begin{proof}
For (1), let $\delta=\ellP(A)$.
If the first empty derivative occurred at a limit $\delta$, the
nonempty closed subsets $\partial_\gamma\widehat A$, $\gamma<\delta$,
would have the finite intersection property inside the compact
space $\widehat A$. Their intersection could not be empty.
Since $A\ne0$, we therefore have $\delta=\alpha+1$.
The last derivative $F=\partial_\alpha\widehat A$ is nonempty and
compact. Its dimension function is locally constant, so the
disjoint open level sets form a cover of $F$ with a finite
subcover. The dimension function on $F$ thus has a finite maximum
$d$. Equation~\eqref{eq:last-stratum} gives the rank formula.

For (2), Theorem~\ref{thm:ccr} makes every Fell point rank countable.
For each $t\in\widehat A$, choose a section $(U_t,a_t)$ whose rank
equals $\rkF(t,A)$; the defining infimum is attained.
Every $s\in U_t$ has $\rkF(s,A)\leq r(U_t,a_t)=\rkF(t,A)$.
Compactness gives a finite subcover $U_{t_1},\ldots,U_{t_k}$, whence
\[
 \rkF(A)=\max_{1\leq i\leq k}\rkF(t_i,A).
\]
Corollary~\ref{cor:fell-point-form} gives the required form.
The primitive spectrum of a unital algebra is compact
\cite[II.6.5.7]{blackadar_operator_2006}.
\end{proof}

The existence of a largest liminary ideal is classical
\cite[IV.1.3.9]{blackadar_operator_2006}.
The next proposition describes it in terms of Fell-rank bounds.

\begin{proposition}\label{prop:fell-ideals}
Let $A$ be type I and $\gamma<\omega_1$.
The open set $F_\gamma(A)$ corresponds to the largest ideal
$I_\gamma(A)$ of Fell rank at most $\gamma$.
This ideal is generated by the positive elements $b$ with finite
weight and $r(w_b)\leq\gamma$.
The closed ideal generated by all $I_\gamma(A)$ is the largest
liminary ideal of $A$.
\end{proposition}
\begin{proof}
Fell locality proves the first assertion.
An element with the stated rank bound has its open support in
$F_\gamma(A)$. Conversely, a section of rank at most $\gamma$
can be replaced, by Lemma~\ref{lem:hereditary-section}, with an
element whose support and restricted weight are exactly those
of the section. These supports cover $F_\gamma(A)$.

Each $I_\gamma(A)$ is liminary by Theorem~\ref{thm:ccr}.
A closed sum of liminary ideals is liminary by the representation
argument in that proof. Every liminary ideal has a countable
Fell bound and lies in some $I_\gamma(A)$ by locality.
\end{proof}

\Needspace{14\baselineskip}
\section{Finite ranks and dimension drops}\label{sec:finite}

A bound on representation dimension need not prevent dimension drops,
and ranks of individual elements can behave differently from the
dimension function. The finite-rank criterion and the examples below
separate these phenomena, clarifying exactly what a finite Pedersen
bound imposes.

\begin{proposition}\label{prop:finite-pedersen}
For nonzero preliminary $A$ and $n\geq1$,
\[
 \rkP(A)\leq n
 \quad\Longleftrightarrow\quad
 \ellP(A)=1\text{ and }\dim(\pi)\leq n
             \text{ for every irreducible }\pi.
\]
These conditions are also equivalent to an algebra decomposition
\[
 A\cong\bigoplus_{d=1}^{n}A_d,
 \qquad A_d=0\text{ or }A_d\text{ is }d\text{-homogeneous}.
\]
In particular, nonzero commutative algebras have Pedersen rank one.
If all irreducible dimensions are at most $n$, then
\[
 \ellP(A)\leq n,\qquad \rkP(A)\leq\omega(n-1)+1.
\]
\end{proposition}
\begin{proof}
Apply Proposition~\ref{prop:finite-range} to $w_A$ for the first
equivalence. If $w_A$ is locally constant and bounded by $n$, its
level sets $X_d=\{t:w_A(t)=d\}$, $1\leq d\leq n$, are disjoint
clopen subsets covering $\widehat A$. Their corresponding ideals
$A_d$ give the stated finite direct sum, and each nonzero $A_d$
is $d$-homogeneous. Conversely, such a decomposition makes $w_A$
locally constant and bounded by $n$.
For the last bound use the list $1,\ldots,n$, allowing unused values.
\end{proof}

We use $n$-subhomogeneous in its classical sense of bounding all
irreducible dimensions by $n$. A finite bound on Pedersen
rank requires local constancy as well.

\begin{example}\label{ex:referee}
Let $X=\{0\}\cup\{1/n:n\geq1\}$ and $A=C(X,M_2)$.
Its dimension function is constant $2$, giving
\[
 \ellP(A)=1,\qquad \rkP(A)=2,\qquad \rkF(A)=1.
\]
For $a(t)=\operatorname{diag}(1,t)$, however,
$w_a(0)=1$ and $w_a(1/n)=2$.
Its first derivative is $\{0\}$ and its second is empty, so
\[
 o(w_a)=2,\qquad r(w_a)=\omega+1.
\]
Accordingly, $\Her(a)$ has Pedersen length $2$ and rank $\omega+1$.
This distinguishes the algebra invariant from an element-weight
rank. It also shows that a hereditary subalgebra can have larger
Pedersen rank than its ambient algebra.
\end{example}

\begin{example}\label{ex:scalar-drop}
For the same $X$ let
\[
 B=\{f\in C(X,M_2):f(0)\in\C1_2\}.
\]
The irreducibles are the two-dimensional evaluations at $1/n$ and
the scalar character at $0$. Its spectrum is $X$: central scalar
functions identify its base points, and the ideal vanishing at $0$
is $\bigoplus_nM_2$.
Thus
\[
 \ellP(B)=2,\qquad \rkP(B)=\omega+1.
\]
This unital algebra is classically $2$-subhomogeneous.
Its Fell rank is also $\omega+1$:
a positive element nonzero at $0$ is invertible in all sufficiently
late matrix fibres, so its weight there is $2$, whereas its weight
at $0$ is $1$. Each isolated point admits a rank-one section.
\end{example}

\begin{remark}\label{rem:product}
The inequality
$r(w_{bc})\leq\min\{r(w_b),r(w_c)\}$ can fail even for commuting
positive elements $b,c$ with finite weights, so that $bc$ is also
positive and has finite weight.
In Example~\ref{ex:referee}, take $b=a$ and $c=1_A$. Then
$r(w_{bc})=\omega+1>2=r(w_c)$.
Proposition~\ref{prop:element-hereditary} instead compares an
element with its own generated hereditary algebra.
\end{remark}

\Needspace{14\baselineskip}
\section{Explicit realizations of countable ranks}\label{sec:realization}

Having restricted the possible ordinal values, we construct algebras
that attain them. Countable direct sums and matrix amplification provide
the elementary operations, while scalar unitization creates successive
derivative stages. A final construction separates Pedersen complexity
from Fell rank by keeping local rank-one sections available.

\subsection{Direct sums and matrix amplification}

The recursive constructions require control of the invariants under
two elementary operations. Direct sums assemble ordinal suprema, while
matrix amplification changes representation dimensions without changing
the Pedersen derivative. These formulas let us adjust the ranks of the
later examples.

\begin{lemma}\label{lem:sums}
For countable families of preliminary algebras,
\[
 \ellP\!\left(\bigoplus_iA_i\right)=\sup_i\ellP(A_i),\qquad
 \rkP\!\left(\bigoplus_iA_i\right)=\sup_i\rkP(A_i),\qquad
 \rkF\!\left(\bigoplus_iA_i\right)=\sup_i\rkF(A_i).
\]
Matrix amplification preserves the Pedersen derivative and length.
If $\ellP(A)=\alpha+1$ and the last stratum has maximum dimension
$m<\infty$, then
\[
 \rkP(M_d(A))=\omega\alpha+dm\qquad(d\geq1).
\]
\end{lemma}
\begin{proof}
The spectrum of a direct sum is the disjoint union of its summand
spectra. Use open locality, and Fell locality for the third equality.
Matrix amplification multiplies the dimension weight by $d$.
This injective transformation preserves all local-constancy loci
and hence all derivatives. Apply \eqref{eq:last-stratum}.
\end{proof}

\subsection{Scalar unitization}

Scalar unitization adds a distinguished character to a countable direct
sum. We arrange the dimensions in the summands so that this character
survives to a prescribed derivative stage and forces the required lower
bound on Fell rank, yielding the transfinite examples.

For each $0<\alpha<\omega_1$ choose a nondecreasing sequence
$\beta_n<\alpha$ with
\begin{equation}\label{eq:cofinal-choice}
 \sup_n(\beta_n+1)=\alpha.
\end{equation}
For $\alpha=\beta+1$ use the constant sequence $\beta_n=\beta$;
at a limit use a cofinal sequence. Define unital algebras recursively:
\begin{equation}\label{eq:realization-algebras}
 A_0=\C,\qquad D_n=M_2(A_{\beta_n}),\qquad
 A_\alpha=\left(\bigoplus_{n\geq1}D_n\right)^\sim\quad(\alpha>0),
\end{equation}
where the superscript $\sim$ denotes scalar unitization.
For $\alpha>0$, let $\infty_\alpha$ denote the quotient character,
and let $\infty_0$ be the unique character of $A_0=\C$.

\begin{lemma}\label{lem:unitization-spectrum}
Suppose the nonzero unital preliminary algebras $D_n$ have compact
Hausdorff spectra $X_n$. Then the spectrum of
$(\bigoplus_nD_n)^\sim$ is the one-point compactification
\[
 X=\{\infty\}\cup\bigsqcup_nX_n.
\]
The dimension at $\infty$ is $1$, and on $X_n$ it is the dimension
function of $D_n$. Every neighborhood of $\infty$ contains all
but finitely many whole components $X_n$.
\end{lemma}
\begin{proof}
An irreducible representation either kills the direct-sum ideal,
and is the quotient character, or restricts nontrivially to
exactly one summand and is the unital extension of one of its
irreducibles. Each $X_n$ is compact and clopen.
Elements have components $\lambda1_{D_n}+x_n$, with
$\norm{x_n}\to0$. If $\lambda\ne0$, these components are invertible
for all sufficiently large $n$.
Basic open neighborhoods of the quotient character therefore contain
all sufficiently late whole components.
Conversely, the complement of a finite union of components is open,
using the sum of their central unit projections.
Together with the component topologies these are exactly the
one-point compactification neighborhoods.
\end{proof}

\begin{theorem}\label{thm:successor-realization}
For each $\alpha<\omega_1$, $A_\alpha$ is a separable unital
preliminary AF algebra with countable compact Hausdorff spectrum,
and
\[
 \ellP(A_\alpha)=\alpha+1,\qquad
 \rkP(A_\alpha)=\rkF(A_\alpha)=\omega\alpha+1.
\]
Moreover,
\[
 \partial_\alpha(\widehat{A_\alpha},w_{A_\alpha})
       =\{\infty_\alpha\}.
\]
Its distinguished character has point Pedersen order $\alpha$
and Fell rank $\omega\alpha+1$. These algebras are scattered.
\end{theorem}
\begin{proof}
Use transfinite induction, including the displayed singleton-derivative
identity in the induction hypothesis.
The assertion for $A_0=\C$ is immediate.
Matrix amplification, countable direct sums, and scalar
unitization preserve AF algebras.
Lemma~\ref{lem:unitization-spectrum} proves preliminarity and the
assertions about the spectrum at the next stage.
This compact second countable Hausdorff spectrum is metrizable.
For a selfadjoint element $h$ in a unital C*-algebra, every
$\lambda\in\operatorname{sp}(h)$ is an eigenvalue in some irreducible
representation: extend evaluation at $\lambda$ on $C^*(1,h)$ to a pure
state of the algebra and use its irreducible GNS representation.
The cyclic vector is a $\lambda$-eigenvector, since the state vanishes
on $(h-\lambda1)^2$. Representations preserve inverses, so
$\operatorname{sp}(h)$ is exactly the union of the spectra of its
irreducible images. Here this is a countable union of finite sets.
Scatteredness is equivalent to countability
of the spectrum of each selfadjoint element
\cite[Theorem~2.2\textup{(vii)}]{hida_koszmider_2020}, so
these algebras are scattered.

Put $X=\widehat{A_\alpha}$ and $X_n=\widehat{D_n}$.
The dimension weight on $X_n$ is twice that of
$\widehat{A_{\beta_n}}$. By induction its length is $\beta_n+1$,
and its last stratum is the single distinguished point, now of
dimension $2$. By locality
$\partial_\gamma X\cap X_n=\partial_\gamma X_n$.
For $\gamma<\alpha$, infinitely many of these component derivatives
are nonempty, by \eqref{eq:cofinal-choice} and monotonicity of
$(\beta_n)$. All their dimensions are at least $2$.
Every neighborhood of $\infty_\alpha$ contains all but finitely
many components, so the dimension weight, of value $1$ at
$\infty_\alpha$, is not locally constant there in
$\partial_\gamma X$. Induction keeps that point through stage
$\alpha$. At that stage all the component derivatives are empty:
\[
 \partial_\alpha X=\{\infty_\alpha\},\qquad
 \partial_{\alpha+1}X=\varnothing.
\]
This proves the length and Pedersen-rank formulas.

For Fell rank, take any section $(U,a)$ at $\infty_\alpha$.
Its components are $\lambda1_{D_n}+x_n$, with $\lambda>0$ and
$\norm{x_n}\to0$, so they are positive invertible for all sufficiently
large $n$. There is a neighborhood $V\subseteq U$ consisting of
$\infty_\alpha$ and all sufficiently late whole components.
On $V$ we have $w_a=w_{A_\alpha}$.
Consequently locality gives
\[
 r(U,a)\geq r(V,a)\geq
 r_{(V,w_{A_\alpha})}(\infty_\alpha)=\omega\alpha+1.
\]
This holds for every section.
The upper bound follows from
$\rkF(A_\alpha)\leq\rkP(A_\alpha)$.
\end{proof}

\begin{corollary}\label{cor:all-pedersen}
Every nonzero countable ordinal is the Pedersen rank of a separable
preliminary AF algebra with countable Hausdorff spectrum.
Every nonzero countable successor is also a Pedersen length of a
unital such algebra. Every nonzero countable limit length occurs
on a countable direct sum of these algebras.
\end{corollary}
\begin{proof}
Write a countable successor as $\eta=\omega\alpha+d$, with
$1\leq d<\omega$. Then $M_d(A_\alpha)$ has last dimension $d$,
so has Pedersen rank $\eta$ by Lemma~\ref{lem:sums}.
For nonzero countable limit $\eta$, choose successor ordinals
$\eta_n<\eta$ cofinal in it, realize each as above, and take their
direct sum. Its rank is $\eta$.
Apply the same direct-sum argument to lengths
$\ellP(A_\alpha)=\alpha+1$.
\end{proof}

\begin{corollary}\label{cor:all-fell}
The nonzero countable Fell ranks are exactly those in
Theorem~\ref{thm:intro-realization}, and all are realized by
separable preliminary AF algebras with countable Hausdorff spectrum.
\end{corollary}
\begin{proof}
Corollary~\ref{cor:fell-point-form} gives the restriction.
Theorem~\ref{thm:successor-realization} realizes $\omega\alpha+1$.
For nonzero countable limit $\lambda$, choose $\alpha_n<\lambda$
cofinal in it. The direct sum $\bigoplus_nA_{\alpha_n}$ has Fell
rank $\sup_n(\omega\alpha_n+1)=\omega\lambda$.
\end{proof}

\begin{corollary}[Unital realizations]\label{cor:unital-realizations}
Among nonzero separable unital preliminary algebras, the possible
Pedersen lengths and Pedersen ranks are exactly the nonzero countable
successor ordinals. Among nonzero separable unital liminary algebras,
the possible Fell ranks are exactly
\[
 \{\omega\alpha+1:\alpha<\omega_1\}.
\]
Every one of these values is realized by a unital preliminary AF
algebra with countable compact Hausdorff spectrum.
\end{corollary}
\begin{proof}
Proposition~\ref{prop:compact-spectrum} proves the restrictions.
For each $\alpha<\omega_1$, Theorem~\ref{thm:successor-realization}
gives the length $\alpha+1$ and Fell rank $\omega\alpha+1$ on
$A_\alpha$. Every nonzero countable successor is
$\omega\alpha+d$ with $1\leq d<\omega$, and is the Pedersen rank
of $M_d(A_\alpha)$ by Lemma~\ref{lem:sums}.
These examples remain unital preliminary AF algebras with the
asserted spectra.
\end{proof}

\Needspace{14\baselineskip}
\subsection{Large Pedersen rank with Fell rank one}

The preceding examples often make the two ranks coincide. We now
separate their roles by realizing dimension functions in a hereditary
subalgebra of a continuous field of compact operators. Local rank-one
sections keep Fell rank one, while the dimension derivative can retain
arbitrarily large countable rank.

\begin{proposition}\label{prop:fell-one-realization}
Let $X$ be a nonempty locally compact second countable Hausdorff
space and $w:X\to\{1,2,\ldots\}$ a lower semicontinuous function.
There is a preliminary hereditary subalgebra $B$ of
$C_0(X)\otimes\K$ with spectrum $X$, dimension function $w$,
and Fell rank one.
If $X$ is zero-dimensional, $B$ is AF.
\end{proposition}
\begin{proof}
For $k\geq1$ let $U_k=\{x:w(x)\geq k\}$.
Choose $f_k\in C_0(X)_+$, positive exactly on $U_k$, with
$\norm{f_k}\leq2^{-k}$. Such functions exist because $U_k$
is locally compact and second countable: take a strictly positive
function in $C_0(U_k)$ and extend it by zero.
The diagonal field
\[
 a(x)=\operatorname{diag}(f_1(x),f_2(x),\ldots)
\]
belongs to $C_0(X)\otimes\K$ by uniform convergence of its finite
diagonal truncations, and has rank $w(x)$ at $x$.
Let $B=\Her(a)$.
The hereditary representation correspondence identifies its
spectrum with $X$ and its dimension function with $w$.

At $x_0$, choose a constant vector $v$ with $a(x_0)v\ne0$ and
$\varphi\in C_c(X)_+$ positive at $x_0$.
The element
$b(x)=\varphi(x)a(x)|v\rangle\langle v|a(x)$ belongs to $B$.
It has rank one on an open neighborhood of $x_0$.
Thus every point has Fell rank one.
Suppose now that $X$ is zero-dimensional. Let $(e_{ij})$ be the
standard matrix units of $\K$. We first identify
\[
 B=\overline{\operatorname{span}}
    \{g e_{ij}:g\in C_c(U_{\max(i,j)}),\ i,j\geq1\}.
\]
For such a $g$, both $f_i$ and $f_j$ are positive and bounded away
from zero on its compact support. Dividing $g$ by $f_if_j$ therefore
expresses $g e_{ij}$ as $a c a$ for a compactly supported matrix field
$c$. Conversely, for arbitrary $c\in C_0(X)\otimes\K$, finite diagonal
truncations of $a$ converge in norm.
The coefficients of a corresponding truncation of $a c a$ belong to
$C_0(U_{\max(i,j)})$, and compactly supported functions are dense in
this space. This proves the identity.

Given finitely many fields from the displayed span, approximate all
their coefficients by functions constant on finitely many compact open
sets, with each support contained in its prescribed $U_{\max(i,j)}$.
Take a common finite disjoint compact open refinement
$V_1,\ldots,V_s$ of these sets. For each nonzero cell $V_l$, let
$n_l$ be the largest matrix index of a coefficient present there.
The nesting of the $U_k$ gives $V_l\subseteq U_{n_l}$, so all these
approximants belong to the finite-dimensional subalgebra
\[
 \bigoplus_{l=1}^{s}
 \{1_{V_l}T:T\in M_{n_l}(\C)\}\ \subseteq B.
\]
Choosing the coefficient errors sufficiently small gives simultaneous
norm approximation of the original fields. Hence $B$ is locally AF.
It is separable, so it is AF
\cite[II.8.3.24]{blackadar_operator_2006}.
\end{proof}

\begin{corollary}\label{cor:independent-pedersen}
Every nonzero countable ordinal occurs as the Pedersen rank of
a separable preliminary AF algebra of Fell rank one.
These algebras can have countable Hausdorff spectra.
\end{corollary}
\begin{proof}
The spectra constructed above are zero-dimensional.
Indeed, their components are clopen, and the distinguished
point has a basis of clopen complements of finite unions of
compact components; induction proves the assertion for $A_\alpha$.
Matrix amplification leaves the spectrum unchanged, and countable
disjoint unions preserve zero-dimensionality.
Take the dimension-weighted spectrum of any realization in
Corollary~\ref{cor:all-pedersen} and apply
Proposition~\ref{prop:fell-one-realization}.
Its Pedersen derivative, length, and rank are unchanged.
\end{proof}

The classes $\{A:\rkF(A)\leq\omega\alpha+1\}$ therefore strictly
increase with $\alpha$ and exhaust the separable liminary algebras.
Bounds on the Pedersen rank, or on its length, likewise
exhaust the preliminary algebras. Their finite levels retain the
local-constancy conditions in Proposition~\ref{prop:finite-pedersen}.

\section*{Funding}
This work is supported by ERC grant \textquotedblleft Definable
Algebraic Topology\textquotedblright\ (DAT, grant agreement
No.~101077154; project DOI
\href{https://doi.org/10.3030/101077154}{10.3030/101077154}).
Funded by the European Union. Views and opinions expressed are
those of the author only and do not necessarily reflect those
of the European Union or the European Research Council Executive
Agency. Neither the European Union nor the granting authority can
be held responsible for them. The author was also partially
supported by the Gruppo Nazionale per le Strutture Algebriche,
Geometriche e le loro Applicazioni (GNSAGA) of the Istituto
Nazionale di Alta Matematica (INdAM), and the University of Bologna.

\medskip
\noindent\textit{Use of artificial intelligence.}
ChatGPT by OpenAI, OpenAI Codex, and Claude by Anthropic were used
during September--November 2026 in the preparation and revision
of this manuscript for exploration, source retrieval, editorial
and structural revision, and assistance in checking proofs,
notation, terminology, internal consistency, cross-references,
and bibliographic citations.

\end{document}